\documentclass[12pt,a4paper]{article}
\usepackage{mathrsfs}
\usepackage{epsfig, graphicx}
\usepackage{latexsym,amsfonts,amsbsy,amssymb}
\usepackage{amsmath,amsthm}
\usepackage{color}
\usepackage[colorlinks, citecolor=blue]{hyperref}
\makeatletter 

\@addtoreset{equation}{section}
\makeatother 
\allowdisplaybreaks 
\newtheorem{theorem}{Theorem}[section]
\newtheorem{lemma}{Lemma}[section]
\newtheorem{corollary}{Corollary}[section]
\newtheorem{remark}{Remark}[section]
\newtheorem{algorithm}{Algorithm}[section]

\newtheorem{example}{Example}[section]
\begin{document}
\title{A Multigrid Method for the Ground State Solution of
Bose-Einstein Condensates\footnote{This work was supported in part by National Science Foundations of China
(NSFC 91330202, 11371026, 11001259, 11031006, 2011CB309703)
and the National Center for Mathematics and Interdisciplinary Science,
CAS and the President Foundation of AMSS-CAS.}
}
\author{
Hehu Xie\footnote{LSEC, ICMSEC,Academy of Mathematics and Systems Science, Chinese Academy of
Sciences, Beijing 100190, China (hhxie@lsec.cc.ac.cn)}\ \ \ and \ \
Manting Xie\footnote{LSEC, ICMSEC,Academy of Mathematics and Systems Science, Chinese Academy of
Sciences, Beijing 100190, China (xiemanting@lsec.cc.ac.cn)}
}
\date{}
\maketitle
\begin{abstract}
A multigrid method is proposed to compute the ground state solution of Bose-Einstein
condensations by the finite element method
 based on the multilevel correction for eigenvalue problems and the multigrid method for
 linear boundary value problems. In this scheme, obtaining the optimal approximation for the
 ground state solution of Bose-Einstein condensates includes a sequence of
 solutions of the linear boundary value problems by the multigrid method on the
 multilevel meshes and a series of solutions of nonlinear eigenvalue problems
  on the coarsest finite element space.
The total computational work of this scheme can reach almost the optimal
order as same as solving the corresponding linear boundary value problem. Therefore,
this type of multigrid scheme can improve the overall efficiency for the simulation of
Bose-Einstein condensations. Some numerical experiments are provided to validate the
efficiency of the proposed method.
\end{abstract}

\section{Introduction}
Recently, Bose-Einstein condensation (BEC), which is a gas of bosons that are in the same quantum state,
is an active field \cite{BaoCai, GrifSnokeStrin,Ketterle}. In 2001, the Nobel Prize in Physics was
awarded Eric A. Cornell, Wolfgang Ketterle and Carl E. Wieman
\cite{AnderEnsherMattewWieman,CornellWieman,Ketterle} for their research in BEC.
The properties of the condensate at zero or very low temperature \cite{DalGioPitaString,LiebSeiYang}
can be described by the well-known Gross-Pitaevskii equation (GPE) \cite{Gross,JinLevermoreMcLaughlin}
which is a time-independent nonlinear Schr\"odinger equation \cite{LaudauLifschitz}.
So far, it is found that the GPE fits well for the most of experiments
\cite{AnglinKetterle,Cornell,DalGioPitaString,HauBuschLiuDutton}.

As we know that the wave function $\psi$ of a sufficiently dilute condensates satisfies
the following GPE
\begin{eqnarray}\label{GPE}
\left(-\frac{\hbar^2}{2m}\Delta + \widetilde{W} + \frac{4\pi\hbar^2aN}{m}|\psi|^2\right)\psi &=&\mu\psi,
\end{eqnarray}
where $\widetilde{W}$ is the external potential, $\mu$ is the chemical potential and $N$ is the number
of atoms in the condensate. The effective two-body interaction is $4\pi\hbar^2a/m$, where $\hbar$ is the
Plank constant, $a$ is the scattering length (positive for repulsive interaction and negative for
attractive interaction) and $m$ is the particle mass. In this paper, we assume the external potential
 $\widetilde{W}(x)$ is measurable and locally bounded and tends to infinity as $|x|\rightarrow\infty$
in the sense that
\begin{eqnarray*}
\inf_{|x|\geq r}\widetilde{W}(x)\rightarrow \infty\ \ \ \ {\rm for}\ r\rightarrow \infty.
\end{eqnarray*}
Then the wave function $\psi$ must vanish exponentially fast as $|x|\rightarrow \infty$.
Furthermore (\ref{GPE}) can be written as
\begin{eqnarray}\label{GPE_Simple}
\left(-\Delta +\frac{2m}{\hbar^2}\widetilde{W}+8\pi aN|\psi|^2\right)\psi&=&
\frac{2m\mu}{\hbar^2}\psi.
\end{eqnarray}
Hence in this paper, we are concerned with the following non-dimensionalized GPE problem:\\
Find $(\lambda,u)\in \mathbb{R}\times H_0^1(\Omega)$ such that
\begin{equation}\label{GPEsymply2}
\left\{
\begin{array}{rcl}
-\Delta u + Wu + \zeta |u|^2u &=& \lambda u,\ \ \  {\rm in}\  \Omega,\\
u &=& 0,\ \ \ \ \  {\rm on}\ \partial \Omega,\\
\int_\Omega|u|^2 &=& 1,
\end{array}
\right.
\end{equation}\\where $\Omega \subset \mathbb{R}^d$ $(d = 1,2,3)$ denotes the computing
domain which have the cone property \cite{Adams}, $\zeta$ is some positive constant and
$W(x) = \gamma_1 x^2_1 +\cdots + \gamma_d x^2_d \geq 0$
with  $\gamma_1, \cdots, \gamma_d > 0$ \cite{BaoTang,ZhouBEC}.

So far, there have existed many papers discussing the numerical methods for the
time-dependent GPEs and time-independent GPEs. Please refer to the papers \cite{Adhikari,AdhikariMuruganandam,AnglinKetterle,BaoCai,BaoDu,BaoTang,Cerimele.et.al.,Chiofalo,Cornell,
CornellWieman,DalGioPitaString,Dodd,EdwardsBurnett,Ketterle,SchneiderFeder}
and the papers cited therein. Especially, in \cite{ZhouBEC}, the convergence
of the finite element method for GPEs has been proved and the priori error estimates
of the finite element method for GPEs has been presented in \cite{CancesChakirMaday}
which will be used in this paper.


Recently, a type of multigrid method for eigenvalue problems has been proposed in
 \cite{LinXie,Xie_Steklov,Xie_Nonconforming,Xie_JCP}.
The aim of this paper is to present a multigrid scheme for GPE (\ref{GPEsymply2})
based on the multilevel correction method in \cite{LinXie}.
With this method, solving GPE will not be more difficult than solving
the corresponding linear boundary value problem. The multigrid method for GPE
is based on a series of nested finite element spaces with different level
of accuracy which can be built in the same way as the multilevel method for boundary
value problems \cite{Xu}.
The corresponding error and computational work estimates
of the proposed multigrid scheme for the GPE will also be analyzed.
Based on the analysis, the proposed method can obtain optimal errors with an almost optimal
computational work. The eigenvalue multigrid procedure can be described as follows:
(1) solve the GPE in the initial finite element space;
(2) use the multigrid method to solve an additional linear boundary value problem
which is constructed by using the previous obtained eigenpair approximation;
(3) solve a GPE again on the finite element space which is constructed by combining
the coarsest finite element space with the obtained eigenfunction approximation in
step (2). Then go to step (2) for the next loop until stop. In this method, we replace
solving semi-linear eigenvalue problem GPE on the finest finite element space by
solving a series of linear boundary value problems with multigrid scheme in the
corresponding series of finite element spaces and a series of GPEs in the coarsest
finite element space. So this multigrid method can improve the overall efficiency
of solving GPEs as it does for linear boundary value problems.

An outline of the paper goes as follows. In Section 2, we introduce finite element
method for the ground state solution of BEC, i.e. non-dimensionalized GPE (\ref{GPEsymply2}).
A type of one corrections step is given in Sections 3 based on the fixed-point iteration.
In Section 4, we propose a type of multigrid algorithm for solving the
non-dimensionalized GPE by the finite element method. Section 5 is devoted to
estimating the computational work for the multigrid method defined in Section 4.
Two numerical examples are provided in Section 6 to validate our theoretical analysis.
Some concluding remarks are given in the last section.

\section{Finite element method for GPE problem}

In this section, we introduce some notation and the finite element method
for GPE (\ref{GPEsymply2}). The letter $C$ (with or without subscripts) denotes
a generic positive constant which may be different at its different occurrences.
For convenience, the symbols $\lesssim$, $\gtrsim$ and $\approx$
will be used in this paper. That $x_1\lesssim y_1, x_2\gtrsim y_2$
and $x_3\approx y_3$, mean that $x_1\leq C_1y_1$, $x_2 \geq c_2y_2$
and $c_3x_3\leq y_3\leq C_3x_3$ for some constants $C_1, c_2, c_3$
and $C_3$ that are independent of mesh sizes (see, e.g., \cite{Xu}).
We shall use the standard notation for Sobolev spaces $W^{s,p}(\Omega)$ and their
associated norms $\|\cdot\|_{s,p,\Omega}$ and seminorms $|\cdot|_{s,p,\Omega}$
(see, e.g., \cite{Adams}). For $p=2$, we denote
$H^s(\Omega)=W^{s,2}(\Omega)$ and $H_0^1(\Omega)=\{v\in H^1(\Omega):\ v|_{\partial\Omega}=0\}$,
where $v|_{\partial\Omega}=0$ is in the sense of trace,
$\|\cdot\|_{s,\Omega}=\|\cdot\|_{s,2,\Omega}$. In this paper, we set $V=H_0^1(\Omega)$.
and use $\|\cdot\|_s$ to denote $\|\cdot\|_{s,\Omega}$ for simplicity.

For the aim of finite element
discretization, we define the corresponding weak form for (\ref{GPEsymply2}) as follows:\\
Find $(\lambda,u)\in \mathbb{R}\times V$ such that $b(u,u) = 1$ and
\begin{equation}\label{GPEweakform}
a(u,v) = \lambda b(u,v),\ \ \  \forall v \in V,
\end{equation}
where
\[
a(u,v) := \int_\Omega\big(\nabla u\nabla v + Wuv + \zeta|u|^2 uv\big)d\Omega,
\ \ \ \ b(u,v) := \int_\Omega uvd\Omega.
\]

Now, let us define the finite element method \cite{BrennerScott,Ciarlet}
for the problem (\ref{GPEweakform}). First we
generate a shape-regular decomposition of the computing domain
$\Omega \subset \mathbb{R}^d$ $(d = 2,3)$ into triangles or rectangles for $d = 2$
(tetrahedrons or hexahedrons for $d = 3$) and the diameter of a cell $K \in \mathcal{T}_h$ is
denoted by $h_K$. The mesh diameter $h$ describes the maximum diameter of all cells
$K \in \mathcal{T}_h$. Based on the mesh $\mathcal{T}_h$, we construct the
linear finite element space denoted by $V_h\subset V$. We assume that $V_h\subset V$
is a family of finite-dimensional spaces that satisfy the following assumption:
\begin{equation}
\lim_{h\rightarrow 0}\inf_{v \in V_h} \|w - v\|_1 = 0,\ \ \ \forall w\in V.
\end{equation}

The standard finite element method for (\ref{GPEweakform}) is to solve the following
 eigenvalue problem: \\
Find $(\bar{\lambda}_h,\bar{u}_h)\in \mathbb{R}\times V_h$ such that
 $b(\bar{u}_h,\bar{u}_h) = 1$ and
\begin{equation}\label{GPEfem}
a(\bar{u}_h,v_h) = \bar{\lambda}_h b(\bar{u}_h,v_h),\ \ \   \forall v_h \in V_h.
\end{equation}
Then we define
\begin{equation}\label{delta}
\delta_h(u) := \inf_{v_h\in V_h}\|u - v_h\|_1.
\end{equation}

\begin{lemma}\label{lemma:Maday}
(\cite[Theorem 1]{CancesChakirMaday})
There exists $h_0 > 0$, such that for all $0 < h < h_0$, the smallest eigenpair approximation
 $(\bar{\lambda}_h,\bar{u}_h)$ of (\ref{GPEfem}) having the following error estimates
\begin{eqnarray}
\|u - \bar{u}_h\|_1 &\lesssim& \delta_h(u),\\
\|u - \bar{u}_h\|_0 &\lesssim& \eta_a(V_h)\|u - \bar{u}_h\|_1 \lesssim \eta_a(V_h)\delta_h(u),\\
|\lambda - \bar{\lambda}_h| &\lesssim& \|u - \bar{u}_h\|^2_1
+ \|u - \bar{u}_h\|_0\lesssim \eta_a(V_h)\delta_h(u),
\end{eqnarray}
where $\eta_a(V_h)$ is defined as follows
\begin{eqnarray}\label{eta_a_h}
\eta_a(V_h)=\|u-\bar{u}_h\|_1+ \sup_{f\in L^2(\Omega),\|f\|_0=1}\inf_{v_h\in V_h}\|Tf-v_h\|_1
\end{eqnarray}
with the operator $T$ being defined as follows:\\
 Find $Tf\in u^{\perp}$ such that
\vskip-0.7cm
\begin{eqnarray*}
a(Tf,v)+2(\zeta |u|^2(Tf),v)- (\lambda(Tf),v)=(f,v),\ \ \ \ \forall v\in u^{\perp},
\end{eqnarray*}
and $u^{\perp}=\big\{v\in H_0^1(\Omega): | \int_{\Omega}uvd\Omega=0\big\}$.
\end{lemma}

\section{One correction step based on fixed-point iteration}\label{sec:1correction_fix-point}
In this section, we introduce a type of one correction step based on the fixed-point iteration to
improve the accuracy of the given eigenpair approximation. This correction step contains
solving an auxiliary linear boundary value problem with multigrid method in the finer finite
element space and a GPE on the coarsest finite element space.

In order to define the one correction step, we introduce a very coarse mesh $\mathcal{T}_H$
and the low dimensional linear finite element space $V_H$ defined on the mesh $\mathcal{T}_H$.
Assume we have obtained an eigenpair approximation $(\lambda_{h_k},u_{h_k})\in
\mathbb{R}\times V_{h_k}$ and the coarse space $V_H$ is a subset of $V_{h_k}$.
Now we introduce a type of one correction step to improve the accuracy of
the given eigenpair approximation $(\lambda_{h_k},u_{h_k})$. Let $V_{h_{k+1}} \subset V$
be a finer finite element space of $V_{h_k}$ such that $V_{h_k}\subset V_{h_{k+1}}$.
Based on this finer finite element space, we define the following one correction step.
\begin{algorithm}\label{Algm:One_Step_Correction}
One Correction Step based on Fixed-point Iteration
\begin{enumerate}
\item Define the following auxiliary boundary value problem:\\
Find $\widehat{e}_{h_{k+1}} \in V_{h_{k+1}}$ such that
\begin{eqnarray}\label{Aux_Source_Problem}
(\nabla\widehat{e}_{h_{k+1}},\nabla v_{h_{k+1}}) =
\lambda_{h_k}b(u_{h_k},v_{h_{k+1}}) - a(u_{h_k},v_{h_{k+1}}),\ \ \
 \forall v_{h_{k+1}}\in V_{h_{k+1}}.
\end{eqnarray}
Solve this equation with multigrid method \cite{Bramble,BrennerScott,Hackbush,McCormick,Xu}
to obtain an approximation
$\widetilde{e}_{h_{k+1}} \in V_{h_{k+1}}$ with the error estimate
$\|\widetilde{e}_{h_{k+1}} - \widehat{e}_{h_{k+1}}\|_1 \lesssim \varsigma_{h_{k+1}}$
and set $\widetilde{u}_{h_{k+1}} = u_{h_k} + \widetilde{e}_{h_{k+1}}$. Here $\varsigma_{h_{k+1}}$
is used to denote the accuracy for the multigrid iteration.

\item Define a new finite element space
$V_{H,h_{k+1}} = V_H + {\rm span}\{\widetilde{u}_{h_{k+1}}\}$ and solve the
following eigenvalue problem:\\
Find $(\lambda_{h_{k+1}},u_{h_{k+1}}) \in \mathbb{R} \times V_{H,h_{k+1}}$
such that $b(u_{h_{k+1}},u_{h_{k+1}}) = 1$ and
\begin{eqnarray}\label{simple_Eigen_Problem}
a(u_{h_{k+1}},v_{H,h_{k+1}}) = \lambda_{h_{k+1}}b(u_{h_{k+1}},v_{H,h_{k+1}}),
\ \  \forall v_{H,h_{k+1}} \in V_{H,h_{k+1}}.
\end{eqnarray}
\end{enumerate}
Summarize above two steps into
\begin{eqnarray*}
(\lambda_{h_{k+1}},u_{h_{k+1}}) = Correction(V_H,\lambda_{h_k},u_{h_k},V_{h_{k+1}},\varsigma_{h_{k+1}}).
\end{eqnarray*}
\end{algorithm}
\begin{theorem}\label{Thm:Error_Estimate_One_Step_Correction}
Assume $h_k<h_0$ and there exists a real number $\varepsilon_{h_k}(u)$ such that
the given eigenpair approximation $(\lambda_{h_k},u_{h_k})\in\mathbb{R}\times V_{h_k}$
has the following error estimates
\begin{eqnarray}
\|\bar{u}_{h_k}-u_{h_k}\|_{0}+|\bar{\lambda}_{h_k}-\lambda_{h_k}|
= \varepsilon_{h_k}(u).\label{Estimate_lambda_lambda_h_k}
\end{eqnarray}
Then after one correction step, the resultant approximation
$(\lambda_{h_{k+1}},u_{h_{k+1}})\in\mathbb{R}\times V_{h_{k+1}}$ has the
following error estimates
\begin{eqnarray}
\|\bar{u}_{h_{k+1}}-u_{h_{k+1}}\|_{1} &\lesssim& \varepsilon_{h_{k+1}}(u),
\label{Estimate_u_u_h_{k+1}}\\
\|\bar{u}_{h_{k+1}}-u_{h_{k+1}}\|_{0} &\lesssim& \eta_{a}(V_H)\|u-u_{h_{k+1}} \|_{1},
\label{Estimate_u_u_h_{k+1}_negative}\\
|\bar{\lambda}_{h_{k+1}}-\lambda_{h_{k+1}}|&\lesssim&
\eta_{a}(V_H)\varepsilon_{h_{k+1}}(u),\label{Estimate_lambda_lambda_h_{k+1}}
\end{eqnarray}
where
$\varepsilon_{h_{k+1}}(u):= \eta_{a}(V_{h_k})\delta_{h_k}(u)
+\|\bar{u}_{h_k}-u_{h_k}\|_{0}+|\bar{\lambda}_{h_k}-\lambda_{h_k}|
+\varsigma_{h_{k+1}}$.
\end{theorem}
\begin{proof}
First, we define $H^1(\Omega)$ inner-product $\widehat{a}(\cdot,\cdot)$ as
\[
\widehat{a}(w,v) = \int_\Omega\nabla w\nabla v d\Omega ,\ \ \ \forall w,v \in V.
\]
From problems (\ref{GPEfem}) and (\ref{Aux_Source_Problem}), inequality
(\ref{Estimate_lambda_lambda_h_k}), Lemma \ref{lemma:Maday}, H\"{o}lder inequality and
Sobolev space embedding inequality, the following estimates hold for any
$v_{h_{k+1}} \in V_{h_{k+1}}$
\begin{eqnarray*}
& &\widehat{a}(\bar{u}_{h_{k+1}} - u_{h_k} - \widehat{e}_{h_{k+1}},v_{h_{k+1}})\\
&=& b(\bar{\lambda}_{h_{k+1}} \bar{u}_{h_{k+1}} - \lambda_{h_k}u_{h_k},v_{h_{k+1}})\nonumber\\
&& +
\big((W + \zeta\|u_{h_k}\|^2)u_{h_k} - (W + \zeta\|\bar{u}_{h_{k+1}}\|^2)
\bar{u}_{h_{k+1}},v_{h_{k+1}}\big)\\
&\lesssim& \|\bar{\lambda}_{h_{k+1}} \bar{u}_{h_{k+1}} - \lambda_{h_k}u_{h_k}\|_0\|v_{h_{k+1}}\|_1\nonumber\\
&&\ \ \ +\|\bar{u}_{h_{k+1}}-u_{h_k}\|_0(\|\bar{u}_{h_{k+1}}\|_{0,6,\Omega}^2+\|u_{h_k}\|_{0,6,\Omega}^2)
\|v_{h_{k+1}}\|_{0,6,\Omega}\nonumber\\
&\lesssim& \Big(\|\bar{\lambda}_{h_{k+1}} \bar{u}_{h_{k+1}} -\bar{\lambda}_{h_{k}} \bar{u}_{h_{k}}\|_0
+\|\bar{\lambda}_{h_{k}} \bar{u}_{h_{k}} - \lambda_{h_k}u_{h_k}\|_0\big)\|v_{h_{k+1}}\|_1\nonumber\\
&&\ \ \ \
+ \big(\|\bar{u}_{h_{k+1}}-\bar{u}_{h_{k}}\|_0+\|\bar{u}_{h_{k}}-u_{h_k}\|_0\big)
(\|\bar{u}_{h_{k+1}}\|_1^2+\|u_{h_k}\|_1^2)\|v_{h_{k+1}}\|_1\nonumber\\
&\lesssim& \big(\eta_{a}(V_{h_k})\delta_{h_k}(u)+\varepsilon_{h_k}(u)\big)\|v_{h_{k+1}}\|_1.
\end{eqnarray*}
Then we have
\begin{equation}\label{Projrcton_u_u_k+1}
\|\bar{u}_{h_{k+1}} - u_{h_k} - \widehat{e}_{h_{k+1}}\|_1\lesssim
\eta_{a}(V_{h_k})\delta_{h_k}(u)+\varepsilon_{h_k}(u).
\end{equation}
From (\ref{Projrcton_u_u_k+1}) and $\|\widetilde{e}_{h_{k+1}}
- \widehat{e}_{h_{k+1}}\|_1 \lesssim \varsigma_{h_{k+1}}$, the following estimate holds
\begin{equation}\label{u_uwanhk+1}
\|\bar{u}_{h_{k+1}} - \widetilde{u}_{h_{k+1}}\|_1=\|\bar{u}_{h_{k+1}}-u_{h_k}-\widetilde{e}_{h_{k+1}}\|_1
\lesssim \eta_{a}(V_{h_k})\delta_{h_k}(u)+\varepsilon_{h_k}(u)+\varsigma_{h_{k+1}}.
\end{equation}
Now we come to estimate the error for the eigenpair solution
$(\lambda_{h_{k+1}},u_{h_{k+1}})$ of problem (\ref{simple_Eigen_Problem}). Since $V_{H,h_{k+1}}$ is a
subset of $V_{h_{k+1}}$, we can think of problem (\ref{simple_Eigen_Problem}) as
a subspace approximation for the problem (\ref{GPEfem}).
Then based on the definition of $V_{H,h_{k+1}}$, the subspace approximation result from \cite{CancesChakirMaday}
and  Lemma \ref{lemma:Maday}, the following estimates hold
\begin{eqnarray}\label{u_uhk+1}
\|\bar{u}_{h_{k+1}} - u_{h_{k+1}}\|_1 &\lesssim& \inf_{v_{H,h_{k+1}} \in V_{H,h_{k+1}}}
\|\bar{u}_{h_{k+1}} - v_{H,h_{k+1}}\|
\leq \|\bar{u}_{h_{k+1}} - \widetilde{u}_{h_{k+1}}\|_1 \nonumber\\
&\lesssim&
\eta_{a}(V_{h_k})\delta_{h_k}(u)+\varepsilon_{h_k}(u)+\varsigma_{h_{k+1}}.
\end{eqnarray}
This is the desired result (\ref{Estimate_u_u_h_{k+1}}).
Then (\ref{Estimate_u_u_h_{k+1}_negative}) and (\ref{Estimate_lambda_lambda_h_{k+1}})
can be proved based on (\ref{Estimate_u_u_h_{k+1}}) and Lemma \ref{lemma:Maday}.
\end{proof}

\section{Multigrid method for GPE}
In this section, we introduce a type of multigrid method based on the
{\it One Correction Step} defined in Algorithms \ref{Algm:One_Step_Correction}.
This type of multigrid method can obtain the optimal error estimate as same as
solving the GPE directly on the finest finite element space.

In order to do multigrid scheme, we define a sequence of triangulations
$\mathcal{T}_{h_k}$
of $\Omega$ determined as follows. Suppose $\mathcal{T}_{h_1}$ is
produced from $\mathcal{T}_H$ by the regular refinement and
let $\mathcal{T}_{h_k}$ be obtained from $\mathcal{T}_{h_{k-1}}$ via regular
refinement  such that 
$$h_k\approx\frac{1}{\beta}h_{k-1}, \ \ \ \ k = 2,\cdots,n,$$
where $\beta$ denotes the refinement index.
Based on this sequence of meshes, we construct the corresponding linear finite
element spaces $V_{h_1}, \cdots, V_{h_n}$ such that
\begin{eqnarray}\label{fem_set_relationship}
V_{H} = V_{h_0} \subseteq V_{h_1} \subset V_{h_2} \subset \cdots \subset V_{h_n}\subset V.
\end{eqnarray}
In this paper, we assume the following relations of approximation errors hold
\begin{eqnarray}\label{Error_k_k_1}
\eta_a(V_{h_k})\approx \frac{1}{\beta}\eta_a(V_{h_{k-1}}),\ \ \ \
\delta_{h_k}(u)\approx\frac{1}{\beta}\delta_{h_{k-1}}(u),\ \ \ k=2,\cdots,n.
\end{eqnarray}
\begin{algorithm}\label{Algm:Multi_Correction}
Multigrid Scheme for GPE
\begin{enumerate}
\item Construct a series of nested finite element spaces $V_{H}, V_{h_1}, V_{h_2},
\cdots, V_{h_n}$ such that (\ref{fem_set_relationship}) and (\ref{Error_k_k_1}) hold.

\item Solve the GPE on the initial finite element space $V_{h_1}$:\\
Find $(\lambda_{h_1},u_{h_1})\in \mathbb{R}\times V_{h_1}$ such that $b(u_{h_1},u_{h_1})=1$
and
\begin{eqnarray*}\label{Initial_Nonlinear_Eigen_Problem}
a(u_{h_1},v_{h_1})&=&\lambda_{h_1}b(u_{h_1},v_{h_1}),\ \ \ \ \forall v_{h_1}\in V_{h_1}.
\end{eqnarray*}
\item Do $k = 1, \cdots, n-1$\\
Obtain a new eigenpair approximation
$(\lambda_{h_{k+1}},u_{h_{k+1}})\in \mathbb{R}\times V_{h_{k+1}}$
with the one correction step defined by Algorithm \ref{Algm:One_Step_Correction}
\begin{eqnarray*}
(\lambda_{h_{k+1}},u_{h_{k+1}}) =
{\it Correction}(V_{H},\lambda_{h_k},u_{h_k},V_{h_{k+1}},\varsigma_{h_{k+1}}).
\end{eqnarray*}
end Do
\end{enumerate}
Finally, we obtain an eigenpair approximation $(\lambda_{h_{n}},u_{h_{n}})
\in \mathbb{R}\times V_{h_{n}}$.
\end{algorithm}
\begin{theorem}\label{Thm:Multi_Correction}
Assume $h_1<h_0$ and the error $\varsigma_{h_{k+1}}$ of the linear solving by the multigrid method
in the $k+1$-th level mesh satisfies $\varsigma_{h_{k+1}}\leq \eta_a(V_{h_k})\delta_{h_k}(u)$ for
$k=1,\cdots,n-1$. After implementing Algorithm \ref{Algm:Multi_Correction}, the resultant eigenpair
approximation $(\lambda_{h_n},u_{h_n})$ has the following
error estimates
\begin{eqnarray}
\|\bar{u}_{h_n}-u_{h_n}\|_1 &\lesssim&\beta^2\eta_a(V_{h_n})\delta_{h_n}(u),\label{Multi_Correction_Err_fun1}\\
\|\bar{u}_{h_n}-u_{h_n}\|_0 &\lesssim& \eta_a(V_{h_n})\delta_{h_n}(u),\label{Multi_Correction_Err_fun0}\\
|\bar{\lambda}_{h_n}-\lambda_{h_n}| &\lesssim& \eta_a(V_{h_n})\delta_{h_n}(u),\label{Multi_Correction_Err_eigen}
\end{eqnarray}
with the condition $C\beta^2\eta_a(V_H)<1$ for the concerned constant $C$.
\end{theorem}
\begin{proof}
From Lemma \ref{lemma:Maday} and the definition of Algorithm \ref{Algm:Multi_Correction},
we have $\bar{u}_{h_1}=u_{h_1}$ and $\bar{\lambda}_{h_1}=\lambda_{h_1}$.
Then from the proof of Theorem \ref{Thm:Error_Estimate_One_Step_Correction} with
$\varepsilon_{h_1}(u)=0$ and
 $\varsigma_{h_2}\lesssim\eta_a(V_{h_1})\delta_{h_1}(u)$,
 the following estimates hold
\begin{eqnarray}
\|\bar{u}_{h_2} - u_{h_2}\|_1 &\lesssim& \eta_a(V_{h_1})\delta_{h_1}(u),\label{Multi_uh1_1}\\
\|\bar{u}_{h_2} - u_{h_2}\|_0 &\lesssim& \eta_a(V_H)\|\bar{u}_{h_2} - u_{h_2}\|_1\lesssim
\eta_a(V_H)\eta_a(V_{h_1})\delta_{h_1}(u),\label{Multi_uh1_0}\\
|\bar{\lambda}_{h_2} - \lambda_{h_2}| &\lesssim& \eta_a(V_H)\|\bar{u}_{h_2} - u_{h_2}\|_1\lesssim
\eta_a(V_H)\eta_a(V_{h_1})\delta_{h_1}(u).\label{Multi_lambdah1_0}
\end{eqnarray}
Based on Theorem \ref{Thm:Error_Estimate_One_Step_Correction},
(\ref{Error_k_k_1}), (\ref{Multi_uh1_1})-(\ref{Multi_lambdah1_0})
and recursive argument, the final eigenfunction approximation $u_{h_n}$
has the following error estimates
\begin{eqnarray*}
\|\bar{u}_{h_n} - u_{h_n}\|_1 &\lesssim& \eta_{a}(V_{h_{n-1}}) \delta_{h_{n-1}}(u)+
\|\bar{u}_{h_{n-1}} - u_{h_{n-1}}\|_0 +|\bar{\lambda}_{h_{n-1}}-\lambda_{h_{n-1}}|\\
&\lesssim& \eta_{a}(V_{h_{n-1}}) \delta_{h_{n-1}}(u)+
\eta_a(V_H)\|\bar{u}_{h_{n-1}} - u_{h_{n-1}}\|_1\\
&\lesssim& \eta_{a}(V_{h_{n-1}}) \delta_{h_{n-1}}(u)+
\eta_a(V_H)\eta_{a}(V_{h_{n-2}}) \delta_{h_{n-2}}(u)\nonumber\\
&&\ \ \ \ + \eta_a^2(V_H)\|\bar{u}_{h_{n-2}} - u_{h_{n-2}}\|_1\\
&\lesssim& \sum^{n-1}_{k=1}\big(\eta_{a}(V_H)\big)^{n-k-1}\eta_a(V_{h_k})
\delta_{h_k}(u)\nonumber\\
&\lesssim&\Big(\sum^{n-1}_{k=1}\big(\beta^2\eta_{a}(V_H)\big)^{n-k-1}\Big)
\beta^2\eta_a(V_{h_n})\delta_{h_n}(u)\\
&\lesssim& \frac{1}{1-\beta^2\eta_a(V_H)}\beta^2\eta_a(V_{h_n})\delta_{h_n}(u)
\lesssim \beta^2\eta_a(V_{h_n})\delta_{h_n}(u).
\end{eqnarray*}
This means we have obtained the desired result (\ref{Multi_Correction_Err_fun1}). And
(\ref{Multi_Correction_Err_fun0}) can be proved by the similar
argument in the proof for Theorem \ref{Thm:Error_Estimate_One_Step_Correction}
which can be stated as follows
\begin{eqnarray*}
\|\bar{u}_{h_n} - u_{h_n}\|_0\lesssim \eta_a(V_H)\|\bar{u}_{h_n} - u_{h_n}\|_1
\lesssim \eta_a(V_H)\beta^2\eta_a(V_{h_n})\delta_{h_n}(u)\leq \eta_a(V_{h_n})\delta_{h_n}(u).
\end{eqnarray*}
Similar derivative can lead to the desired result (\ref{Multi_Correction_Err_eigen}) and
the proof is complete.
\end{proof}
Based on the results in Theorem \ref{Thm:Multi_Correction}, we can give the final error estimates for
Algorithm \ref{Algm:Multi_Correction}  as follows.
\begin{corollary}
Under the conditions of Theorem \ref{Thm:Multi_Correction}, we have the following error estimates
\begin{eqnarray}
\|u-u_{h_n}\|_1 &\lesssim&\delta_{h_n}(u),\label{Multi_Correction_Err_fun_final}\\
\|u-u_{h_n}\|_0 &\lesssim& \eta_a(V_{h_n})\delta_{h_n}(u),\label{Multi_Correction_Err_fun_final}\\
|\lambda-\lambda_{h_n}| &\lesssim& \eta_a(V_{h_n})\delta_{h_n}(u).\label{Multi_Correction_Err_eigen_final}
\end{eqnarray}
\end{corollary}
\section{Discussion of the computational work}
In this section, we come to analyze the computational work for the multigrid
 scheme defined in Algorithm \ref{Algm:Multi_Correction}. Since the linear
 boundary value problem (\ref{Aux_Source_Problem})
 in Algorithm \ref{Algm:One_Step_Correction} is solved by multigrid method,
 the computational work for this part is optimal order.

First, we define the dimension of each level linear finite element space as
\begin{eqnarray*}
N_k := {\rm dim}V_{h_k},\ \ \ k=1,\cdots,n.
\end{eqnarray*}
Then we have
\begin{eqnarray}\label{relation_dimension}
N_k \thickapprox\Big(\frac{1}{\beta}\Big)^{d(n-k)}N_n,\ \ \ k=1,\cdots,n.
\end{eqnarray}

The computational work for the second step in Algorithm
\ref{Algm:One_Step_Correction} is different
from the linear eigenvalue problems
\cite{LinXie,Xie_Steklov,Xie_Nonconforming,Xie_JCP}.
 In this step, we need to solve a nonlinear eigenvalue problem
  (\ref{simple_Eigen_Problem}).
Always, some type of nonlinear iteration method (self-consistent iteration or
 Newton type iteration)
is used to solve this nonlinear eigenvalue problem. In each nonlinear iteration
step, we need to
build the matrix on the finite element space $V_{H,h_k}$ ($k=2,\cdots,n$) which
needs the computational
work $\mathcal{O}(N_k)$.
Fortunately, the matrix building can be carried out by the parallel way easily
in the finite element space since it has no data transfer.

\begin{theorem}
 Assume we use $m$ computing-nodes in Algorithm \ref{Algm:Multi_Correction},
the GPE problem solved in the coarse spaces $V_{H,h_k}$ ($k=1,\cdots, n$)
and $V_{h_1}$ need work $\mathcal{O}(M_H)$ and $\mathcal{O}(M_{h_1})$, respectively, and
the work multigrid method for solving the source problem in $V_{h_k}$ be $\mathcal{O}(N_k)$
for $k=2,3,\cdots,n$. Let $\varpi$ denote the nonlinear iteration times when we solve
the nonlinear eigenvalue problem (\ref{simple_Eigen_Problem}).
Then in each computational node, the work involved
in Algorithm \ref{Algm:Multi_Correction} has the following estimate
\begin{eqnarray}\label{Computation_Work_Estimate}
{\rm Total\ work}&=&\mathcal{O}\Big(\big(1+\frac{\varpi}{m}\big)N_n
+ M_H\log N_n+M_{h_1}\Big).
\end{eqnarray}
\end{theorem}
\begin{proof}
Let $W_k$ denote the work in any processor
of the correction step in the $k$-th finite element space $V_{h_k}$.
Then with the correction definition, we have
\begin{eqnarray}\label{work_k}
W_k&=&\mathcal{O}\left(N_k +M_H+\varpi\frac{N_k}{m}\right).
\end{eqnarray}
Iterating (\ref{work_k}) and using the fact (\ref{relation_dimension}), we obtain
\begin{eqnarray}\label{Work_Estimate}
\text{Total work} &=& \sum_{k=1}^nW_k\nonumber =
\mathcal{O}\left(M_{h_1}+\sum_{k=2}^n
\Big(N_k + M_H+\varpi\frac{N_k}{m}\Big)\right)\nonumber\\
&=& \mathcal{O}\Big(\sum_{k=2}^n\Big(1+\frac{\varpi}{m}\Big)N_k
+ (n-1) M_H + M_{h_1}\Big)\nonumber\\
&=& \mathcal{O}\left(\sum_{k=2}^n
\Big(\frac{1}{\beta}\Big)^{d(n-k)}\Big(1+\frac{\varpi}{m}\Big)N_n
+ M_H\log N_n+M_{h_1}\right)\nonumber\\
&=& \mathcal{O}\left(\big(1+\frac{\varpi}{m}\big)N_n
+ M_H\log N_n+M_{h_1}\right).
\end{eqnarray}
This is the desired result and we complete the proof.
\end{proof}
\begin{remark}
Since we have a good enough initial solution $\widetilde{u}_{h_{k+1}}$
in the second step of Algorithm \ref{Algm:One_Step_Correction},
then solving the nonlinear eigenvalue problem (\ref{simple_Eigen_Problem}) always does not
need many nonlinear iteration  times (always $\varpi\leq 3$).
In this case, the complexity in each computational node will be $\mathcal{O}(N_n)$ provided
$M_H\ll N_n$ and $M_{h_1}\leq N_n$.
\end{remark}
\section{Numerical examples}
In this section, we provided two numerical examples to validate the efficiency of the
multigrid method stated in Algorithm \ref{Algm:Multi_Correction}.

\begin{example}\label{Example_1}
In this example, we solve GPE (\ref{GPE}) with the computing domain
$\Omega$ being the unit square
$\Omega=(0,1)\times (0,1)$, $W=x_1^2+x_2^2$ and $\zeta=1$.
\end{example}
The sequence of finite element spaces are constructed by
using the linear finite element on the series of meshes which are produced by regular refinement
with $\beta = 2$ (connecting the midpoints of each edge). In this example, we use two
meshes which are generated by Delaunay method as the initial mesh
$\mathcal{T}_H = \mathcal{T}_{h_1}$ to investigate the convergence behaviors.
Since the exact eigenvalue is not known, we choose an adequately accurate approximation
 as the exact first eigenvalue for our numerical tests.
Figure \ref{Initial_Meshes_Example1} shows the corresponding initial
meshes: one is coarse and the other is fine.

From the error estimate result of GPEs by the finite element method, we have
\begin{eqnarray*}
\delta_h(u) = h,\ \ \ \ \ \eta_a(V_h)=h.
\end{eqnarray*}
Then from Theorem \ref{Thm:Multi_Correction}, the following estimates hold
\begin{eqnarray}\label{Superapproximation}
\|\bar{u}_{h_n}-u_{h_n}\|_1\lesssim h_n^2,\ \ \ \
\|\bar{u}_{h_n}-u_{h_n}\|_0\lesssim h_n^2,\ \ \ \
|\bar{\lambda}_{h_n}-\lambda_{h_n}|\lesssim h_n^2.
\end{eqnarray}
Algorithm \ref{Algm:Multi_Correction} is applied to solve the GPE. For comparison, we
also solve the GPE directly by the finite element method. Figure \ref{GPE_Model_Result}
gives the corresponding numerical results for the ground state solution
(the smallest eigenvalue and the corresponding eigenfunction) corresponding to the two
initial meshes illustrated in
Figure \ref{Initial_Meshes_Example1}. From Figure \ref{GPE_Model_Result},
 we find the multigrid scheme can obtain the optimal error estimates as same as the
direct finite element method for the eigenvalue and the corresponding
eigenfunction approximations which validates the results stated in
Theorem \ref{Thm:Multi_Correction} and (\ref{Superapproximation}).
\begin{figure}[ht]
\centering
\includegraphics[width=6cm,height=6cm]{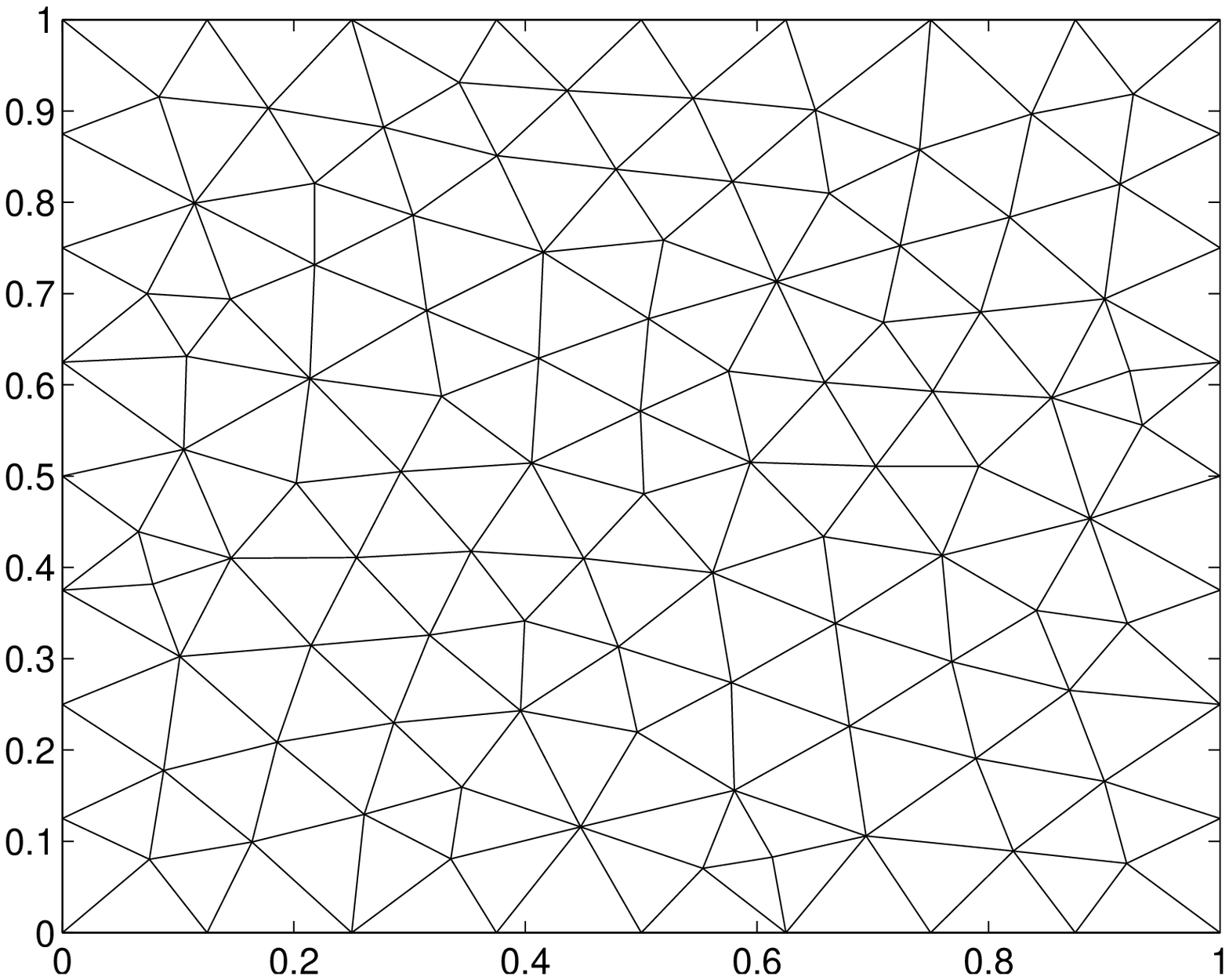}
\includegraphics[width=6cm,height=6cm]{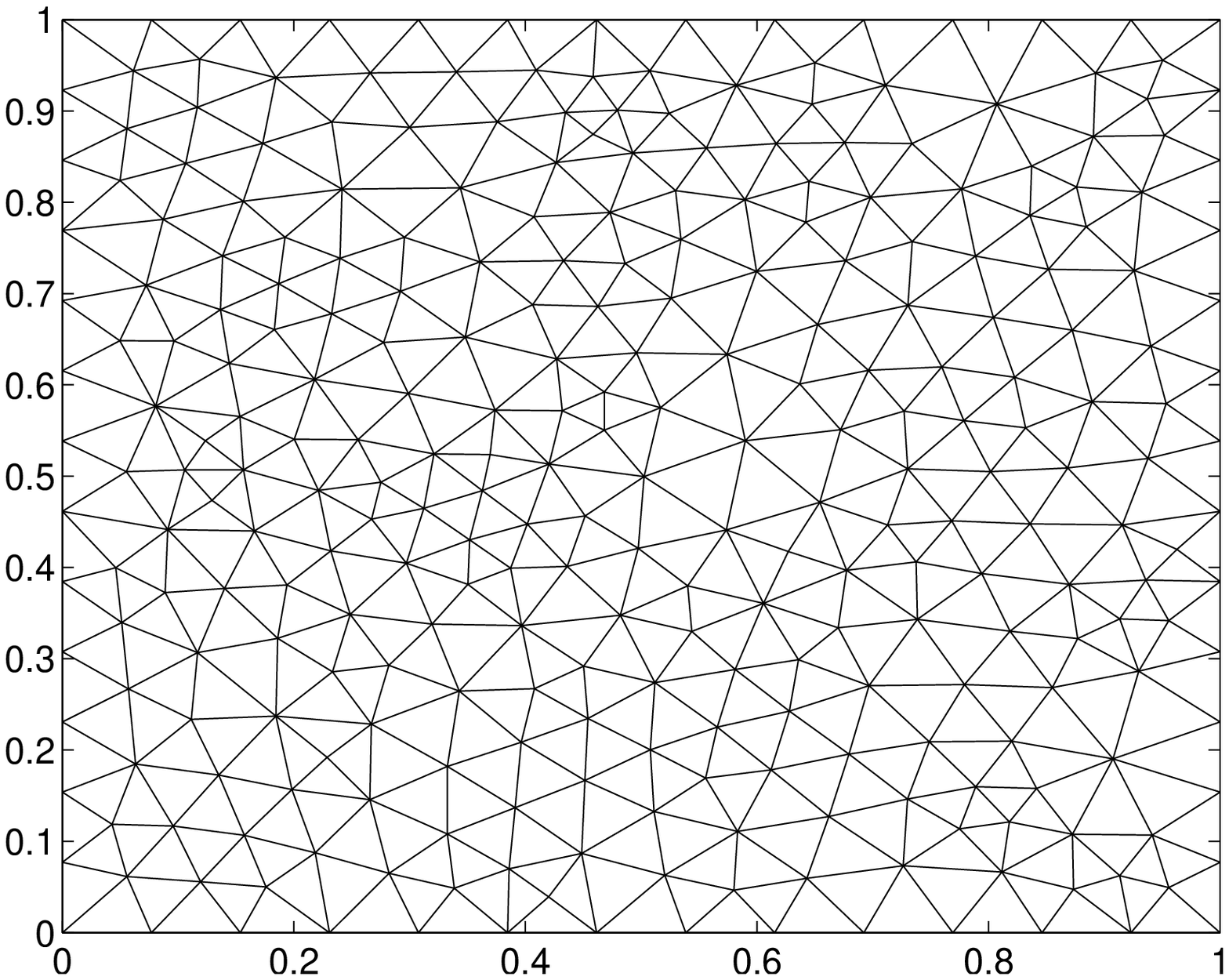}
\caption{\small\texttt The coarse and fine initial meshes for
the unit square (left: $H=1/6$ and right: $H=1/12$)}
\label{Initial_Meshes_Example1}
\end{figure}

\begin{figure}[ht]
\centering
\includegraphics[width=6cm,height=6cm]{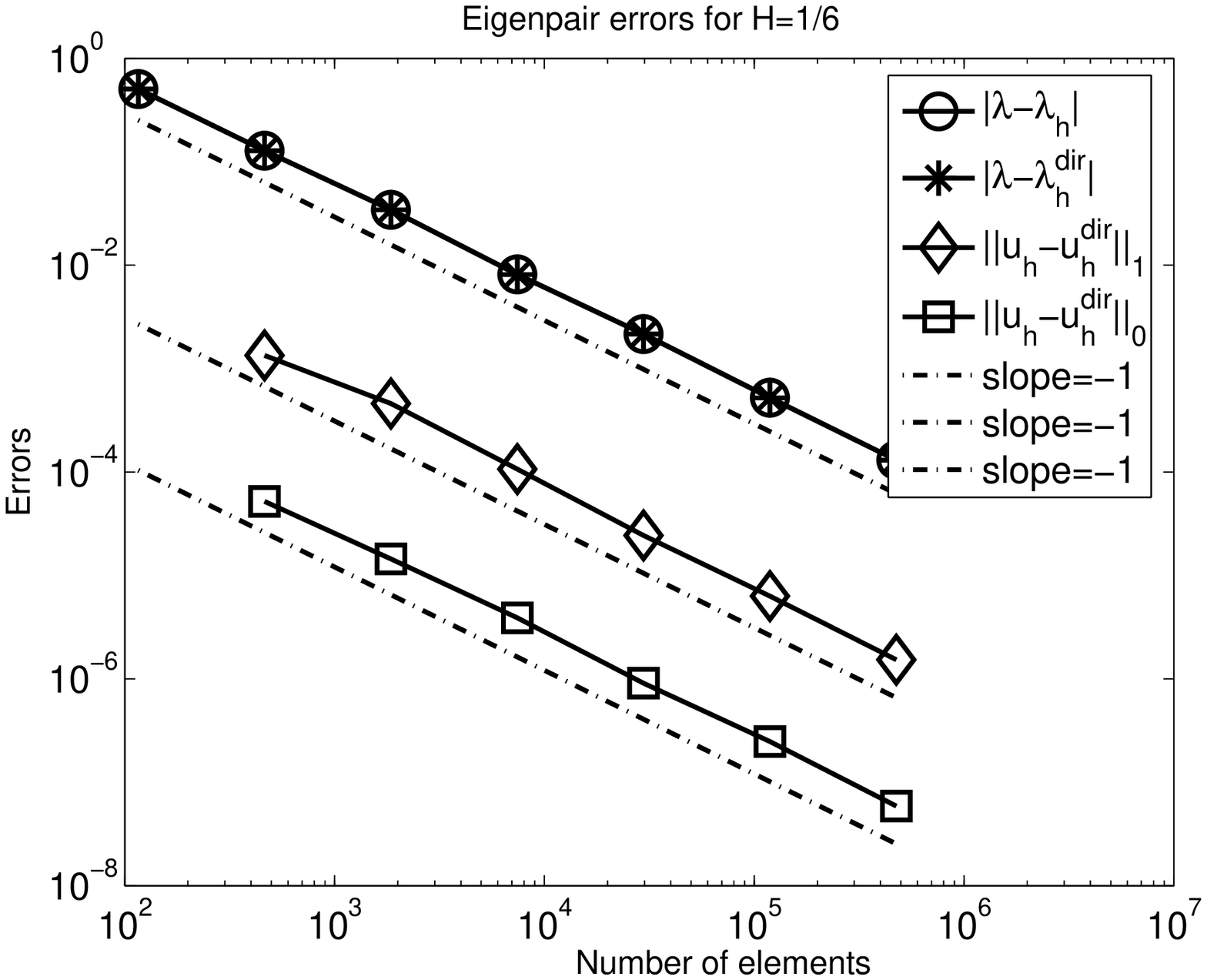}
\includegraphics[width=6cm,height=6cm]{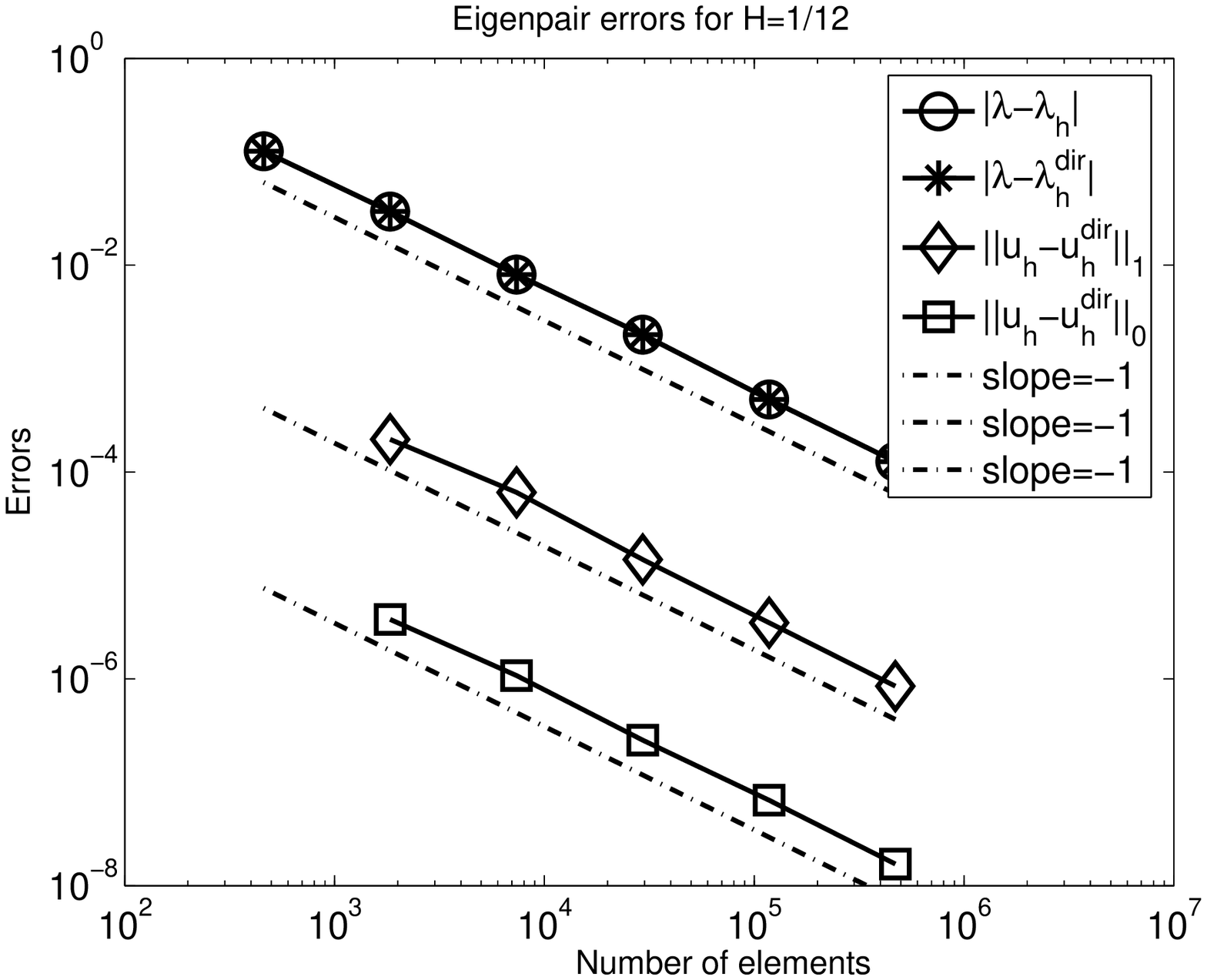}
\caption{\small\texttt The errors of the multigrid
algorithm for the first eigenvalue and the corresponding eigenfunction,
where $u_h^{\rm dir}$ and $\lambda_h^{\rm dir}$ denote the eigenfunction
and eigenvalue approximation by direct eigenvalue solving
(The left figure corresponds to the left mesh in
Figure \ref{Initial_Meshes_Example1} and the right figure corresponds
to the right right mesh in Figure \ref{Initial_Meshes_Example1})}
\label{GPE_Model_Result}
\end{figure}

\begin{example}\label{Example_2}
In this example, we also solve the GPE (\ref{GPE}), where the computing domain $\Omega$ is
the $L$-shape domain $\Omega=(-1,1)\times(-1,1)\backslash[0, 1)\times (-1, 0]$,
$W=x_1^2+x_2^2$ and $\zeta=1$.
\end{example}
Since $\Omega$ has a reentrant corner, eigenfunctions with singularities are expected.
The convergence order for eigenvalue approximations is less than $2$ by the linear
 finite element method which is the order predicted by the theory for regular eigenfunctions.
Thus, the adaptive refinement is adopted to couple with the multigrid method described in
Algorithm \ref{Algm:Multi_Correction} and the ZZ-method \cite{ZienkiewiczZhu}
is used to compute the a posteriori error estimators.

First, we investigate the numerical results for the first eigenvalue approximations.
Since the exact eigenvalue is not known, we also choose an adequately accurate
approximation as the exact smallest eigenvalue for our numerical tests.
We give the numerical results of the multigrid method in which the sequence
of meshes $\mathcal{T}_{h_1}$, $\cdots$, $\mathcal{T}_{h_n}$ is produced by
the adaptive refinement. Figure \ref{GPE_Adaptive_Result} shows the mesh
after $15$ adaptive iterations and
the corresponding numerical results for the adaptive iterations.
From Figure \ref{GPE_Adaptive_Result}, we can find the multigrid method can also
 work on the adaptive family of meshes and obtain the optimal accuracy.
The multigrid method can be coupled with the adaptive refinement naturally which
produce a type of adaptive finite element method (AFEM) for the GPE where
the direct eigenvalue solving in the finest space is not required. This can also
improve the overall efficiency of the AFEM for the nonlinear eigenvalue problem solving.
\begin{figure}[ht]
\centering
\includegraphics[width=6cm,height=6cm]{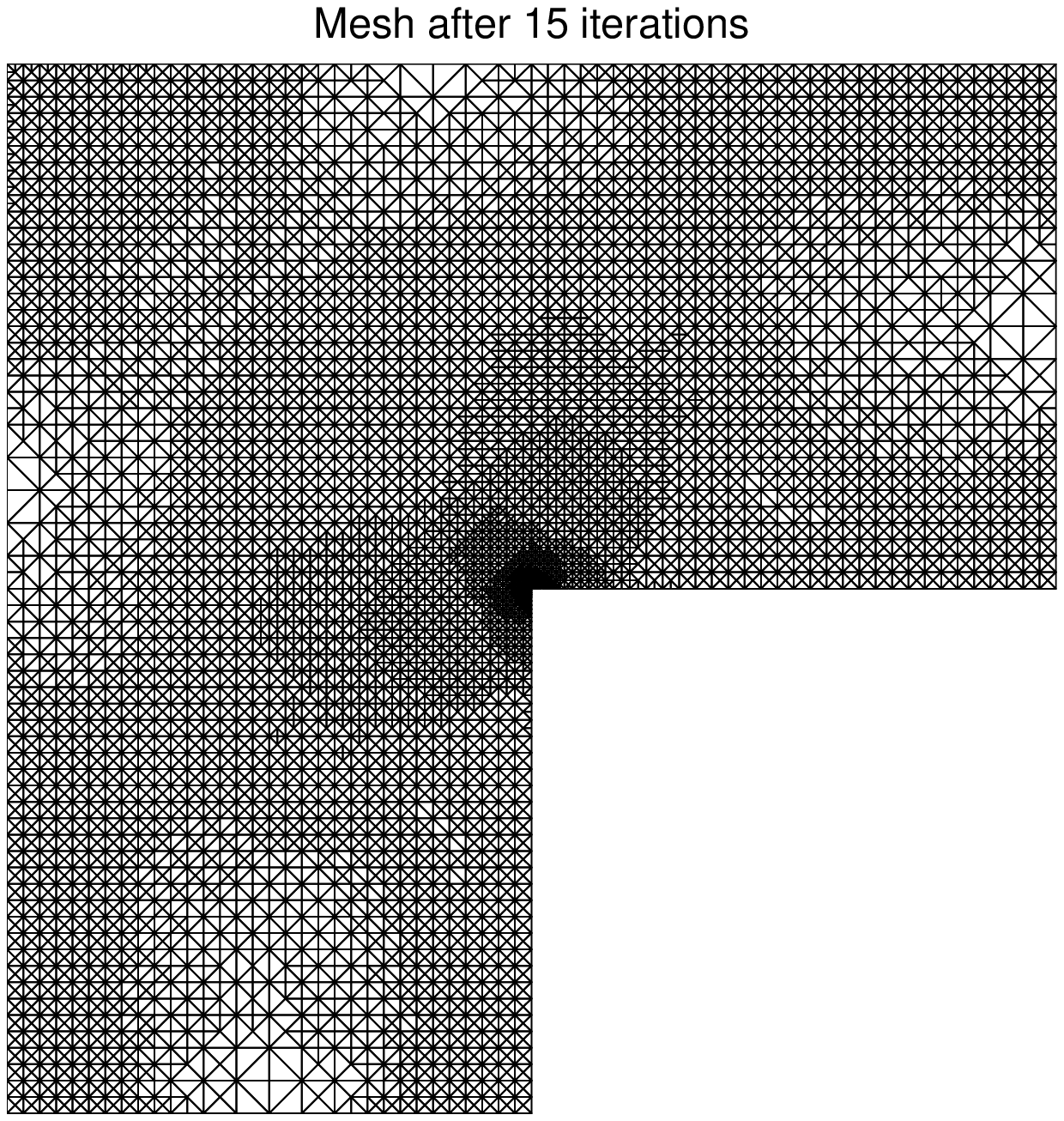}
\includegraphics[width=6cm,height=6cm]{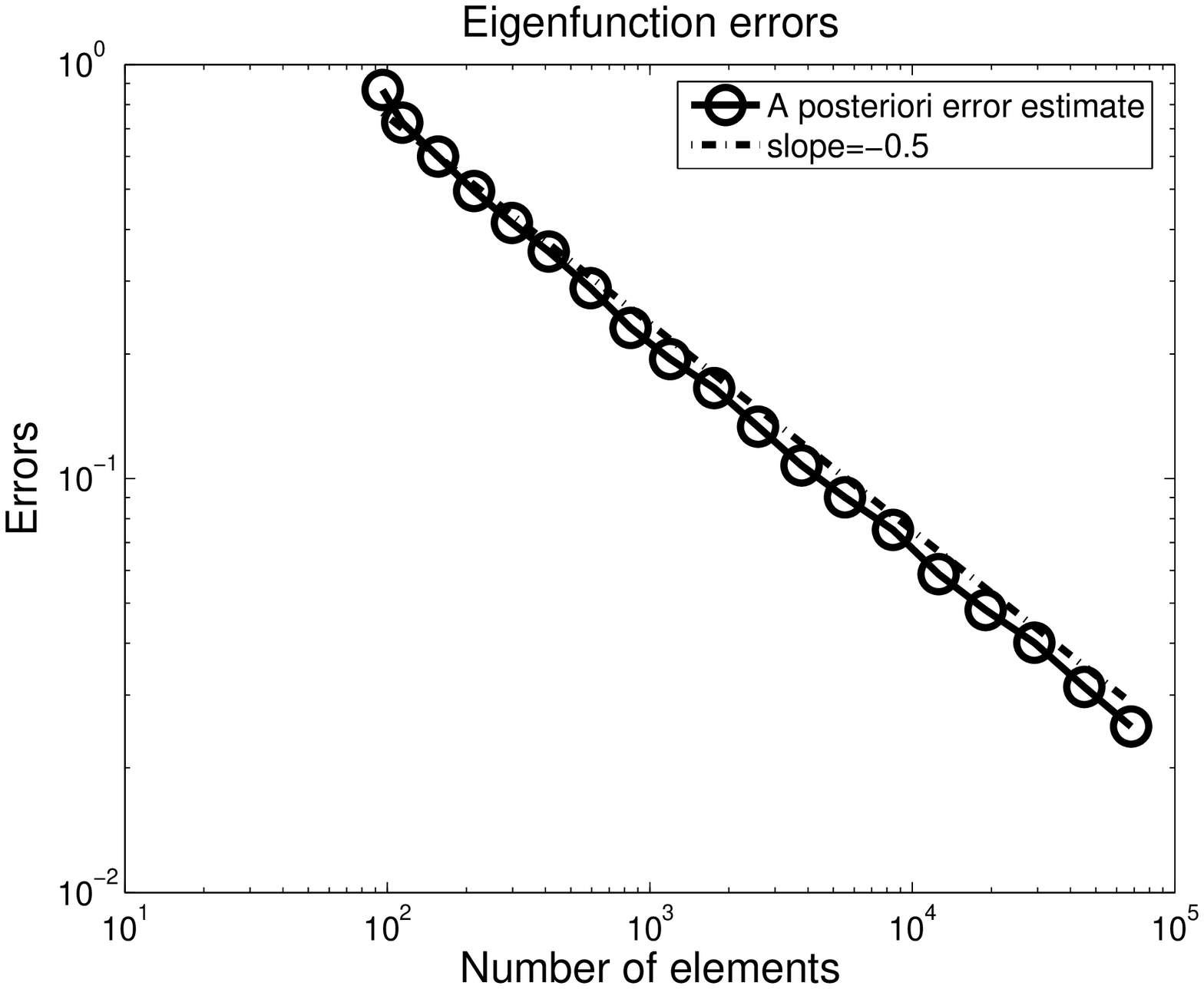}
\caption{\small\texttt The triangulations after adaptive iterations for
Example \ref{Example_2} by the linear element (left) and the
posteriori error estimates for the eigenfunction approximations
 (right)}\label{GPE_Adaptive_Result}
\end{figure}

\section{Concluding remarks}
In this paper, we propose a multigrid method to solve the GPE based on
the multilevel correction method. With this method, solving GPE is not more difficult
than solving the corresponding linear boundary value problem.
The corresponding error and computational work estimate have also been given for the
proposed multigrid scheme.
The idea and the method here can also be extended to other nonlinear eigenvalue problems
which always comes from the electronic structure computation. Algorithm
\ref{Algm:Multi_Correction} can also be coupled with
other numerical schemes to produce some efficient solvers for nonlinear eigenvalue
problems.


\end{document}